%% file: main.tex
\documentclass[12pt, a4paper]{article}

\input{preamble.tex}

\title{Linear isoperimetric inequality for homogeneous Hadamard manifolds}
\author{Hjalti Isleifsson}
\date{\vspace{-36pt}}

\begin{document}

\maketitle

\input{abstract}

\makeatletter{\renewcommand*{\@makefnmark}{}
\footnotetext{{\it Date}: April 6, 2022.}
\footnotetext{Research supported by Swiss National Science Foundation Grant 197090}
\makeatother}

\input{text}

\vspace{12pt}

\input{acknowledgements}

\input{bibliography}
{\small
{\sc Department of Mathematics, ETH Zürich, Rämistrasse 101, 8092 Zürich, Switzerland}\\
\indent
{\it Email address}: hjalti.isleifsson@math.ethz.ch}
\end{document}

%% file: preamble.tex
\usepackage[utf8]{inputenc}
\usepackage[T1]{fontenc}
\usepackage{amsmath,amsthm,amssymb, multicol, array, mathtools}
\usepackage{geometry}
\usepackage{graphicx}
\usepackage{subcaption}
\usepackage{enumitem}
\usepackage{appendix}
\usepackage{titling}

\theoremstyle{theorem}
\newtheorem{theorem}{Theorem}

\newtheorem{lemma}[theorem]{Lemma}

\theoremstyle{definition}

\theoremstyle{remark}
\newtheorem{remark}[theorem]{Remark}

\newcommand{\sL}{\mathcal{L}}

\newcommand{\fg}{\mathfrak{g}}
\newcommand{\fa}{\mathfrak{a}}

\newcommand{\fn}{\mathfrak{n}}
\newcommand{\fk}{\mathfrak{k}}
\newcommand{\fp}{\mathfrak{p}}

\newcommand{\ssm}{\smallsetminus}

\newcommand{\R}{\mathbb{R}}

\newcommand{\Z}{\mathbb{Z}}

\newcommand{\ad}{\operatorname{ad}}
\newcommand{\Ad}{\operatorname{Ad}}

\newcommand{\Exp}{\operatorname{Exp}}

\newcommand{\Zc[1]}{\mathbf{Z}_{#1,\text{c}}}
\newcommand{\Ic[1]}{\mathbf{I}_{#1,\text{c}}}
\newcommand{\Mass}{\mathbf{M}}
\newcommand{\spt}{\operatorname{spt}}

%% file: abstract.tex
\begin{abstract}
It is well known that simply connected symmetric spaces of non-positive sectional curvature admit a linear isoperimetric filling inequality for cycles of dimension greater than or equal to the rank of the space. In this note we extend that result to homogeneous Hadamard manifolds.
\end{abstract}

%% file: text.tex
Let \(X\) be a proper metric space. For \(k\geq 0\) we let \(\Ic[k](X)\) denote the abelian group of \(k\)-dimensional integral metric currents in \(X\) with compact support.  Together with the boundary map \(\partial: \Ic[k+1](X)\rightarrow \Ic[k](X)\), they form a chain complex which generalizes the chain complex of singular Lipschitz chains with integer coefficients. Ignoring some technical details, the reader may safely think of the integral currents as Lipschitz chains. We let \(\Zc[k](X) = \{Z \in \Ic[k](X)\mid \partial Z = 0\}\) denote the subgroup of \(k\)-dimensional cycles. In section 2 in \cite{kl_hrh} is a brief summary of what we will need from the theory of currents. Recall that a {\it Hadamard manifold} is a simply connected complete Riemannian manifold with non-positive sectional curvature. A manifold is {\it homogeneous} if its group of isometries acts transitively. The {\it euclidean rank} of a complete Riemannian manifold is defined as the greatest integer \(r\geq 1\) such that there exists an isometric embedding of \(\R^r\) into the manifold. The aim of this note is to prove the following theorem.

\begin{theorem}\label{main_thm}
Let \(M\) be a homogeneous Hadamard manifold of dimension \(n\) and of euclidean rank \(r\geq 1\). For every \(k\) such that \(r\leq k < n\) there exists a constant \(c_k\) such that for every \(Z\in \Zc[k](M)\) there exists \(V\in \Ic[k+1](M)\) with \(\partial V = Z\) and \(\Mass(V)\leq c_k \Mass(Z)\).
\end{theorem}

In \cite{gromov_asympt}, Gromov stated this result for Lipschitz chains in the case when \(M\) is symmetric. There he also gave a sketch of proof. However, according to the best of my knowledge, there is only one complete proof in the literature which is due to Leuzinger (Theorem 1 in \cite{leuzinger}). We give a sketch of an alternative proof in Remark \ref{symmcase}. In Remark \ref{chains} we outline how Theorem \ref{main_thm} can be proven for Lipschitz chains.\\

We will need several lemmas for the proof, but before turning to them, let us outline the path we are going to take and introduce notation: Let \(M\) be a homogeneous Hadamard manifold of dimension \(n\) and euclidean rank \(r\). Due to the de Rham decomposition theorem, \(M\) can be written as a Riemannian product \(M=M_0\times M_1\) where \(M_0\) is euclidean and \(M_1\) is a homogeneous Hadamard manifold which does not contain a euclidean de Rham factor (see 2.3 in \cite{azencott_wilson_1}). We will apply the structure theory for homogeneous Hadamard manifolds to \(M_1\). According to it (see section 4.1 in \cite{heber}), \(M_1\) is isometric to a Lie group \(G\) with a left invariant metric. Further, \(G = N\rtimes A\) where \(A\) is an abelian Lie subgroup of \(G\) and \(N\) is a nilpotent normal Lie subgroup. We let \(\fa\) and \(\fn\) denote the Lie algebra of \(A\) and \(N\), respectively. The cosets \(nA\), \(n\in N\), are maximal flats of \(G\) which meet \(N\) orthogonally. Let \(\pi_N: G\rightarrow N\) be the projection given by \(\pi_N(s) = n\) for \(s = na\). We are going to construct a set \(S \subseteq G\) such that for every \(\rho > 0\), the set \(S \ssm (N\cdot B(e,\rho))\) contains arbitrarily large balls and such that the projection \(\pi_N\) restricted to \(S\) contracts \(r_1\)-dimensional volume exponentially. Here \(r_1 = r - \dim(M_0)\) denotes the euclidean rank of \(M_1\). Since \(\spt(Z)\) is compact we can assume that it is contained in \(M_0\times S\). Then we will project \(Z\) down to \(M_0\times N\) via 
\[
\pi: M \longrightarrow M_0\times N, \qquad x = (p,s) \longmapsto (p,\pi_N(s))
\]
and use the mapping cylinder plus a suitable filling of the projected cycle as a filling. To make this precise, let \(\sigma: [0,1]\times M\times M\rightarrow M\) denote the geodesic bicombing of \(M\) (i.e. \(t\mapsto \sigma(t,p,q)\) is the geodesic from \(p\) to \(q\)) and let 
\[
\varphi: [0,1]\times M\longrightarrow M, \qquad (t,x) \longmapsto \sigma(t,\pi(x),x).
\] 
Then \(V\coloneqq V_1 + V_2\) is our desired filling where \(V_1 \coloneqq \varphi_\# ([[0,1]]\times Z)\) and \(V_2\) is the suitable filling of \(\pi_\#Z\). This approach is
similar to and inspired by the one proposed by Gromov and the one which Leuzinger takes, in the symmetric space case. Gromov says that one should project onto a maximal flat
and take the mapping cylinder while Leuzinger projects onto a horosphere. The submanifold
\(N\subseteq G\) is in fact the intersection of a family of horospheres (see Theorem 4.2 in \cite{wolter}).\\

Let \(\varphi_G:[0,1]\times G\rightarrow G\), \(\varphi_G(t,x)\coloneqq \sigma_G(t,\pi_N(x),x)\) where \(\sigma_G\) denotes the geodesic bicombing of \(G\). We begin by computing the derivative of \(\varphi_G\).

\begin{lemma}
Let \(s \in G\), \(n\coloneqq \pi_N(s)\) and \(c:[0,1]\rightarrow G\) be the geodesic such that \(c(0) = n\), \(c(1) = s\). Then for any \(v \in T_sG\) there exists a unique \(N\)-Jacobi field \(Y\) along \(c\) such that \(Y(1) = v\). Further, for that Jacobi field we have \(D\varphi_G(t,x)(0,v) = Y(t)\) for every \(t\in [0,1]\).
\end{lemma}

\begin{proof}
It is clear that the normal exponential map of \(N\) is a diffeomorphism so there are no focal points of \(N\) and hence the existence and uniqueness follow. Now, let \(Y\) be an \(N\)-Jacobi field along \(c\) such that \(Y(1) = v\). Then \(Y\) is given by \(Y(t) = D\gamma(0,t)\frac{\partial}{\partial s}\) where \(\gamma\) is a geodesic variation of \(c\) such that \(s\mapsto \gamma(s,0)\) is a curve in \(N\) and for every \(s\) the geodesic \(t\mapsto \gamma(s,t)\) has initial speed orthogonal to \(N\). Since \(t\mapsto \gamma(s,t)\) is a left translation of \(c\) by \(\gamma(s,0)\), it is clear that \(\varphi_G(t,\gamma(s,1)) = \gamma(s,t)\) for all \(s,t\). Differentiating with respect to \(s\) gives:
\[
D\varphi_G(t,x)(0,Y(1)) = D\varphi_G(t,\gamma(0,1))\left(0,D\gamma(0,1)\frac{\partial}{\partial s}\right) = D\gamma(0,t)\frac{\partial}{\partial s} = Y(t)
\]
so we are done.
\end{proof}

For \(g\in G\), let \(L_g : G \rightarrow G\), \(L_g(h) = gh\) denote the left translation of \(G\) by \(g\) and \(R_g: G \rightarrow G\), \(R_g(h) = hg\) denote the right translation of \(G\) by \(g\).

\begin{lemma}\label{jacobi_field_lemma}
Let \(s \in G\), \(n\coloneqq \pi_N(s)\) and \(c:\R\rightarrow G\) be the geodesic such that \(c(0) = n\), \(c(1) = s\); write \(c(t) = n\exp(tH)\) where \(H\in \fa\). Then every \(N\)-Jacobi field along \(c\) can be written as
\[
Y(t) = DL_{c(t)}(e)\left[t\xi + \Ad(\exp(-tH))X\right]
\]
where \(\xi \in \fa\) and \(X \in \fn\). 
\end{lemma}

\begin{proof}
We know that \(Y\) is given by \(Y(t) = D\gamma(0,t)\frac{\partial}{\partial s}\) where \(\gamma\) is a geodesic variation of \(c\) along \(N\). It is clear that \(\gamma\) can be written as \(\gamma(s,t) = \Tilde{n}(s)\exp(tH(s))\) where \(\Tilde{n}\) is a curve in \(N\) with \(\Tilde{n}(0) = n\) and \(\Tilde{H}\) is a curve in \(\fa\) such that \(\Tilde{H}(0) = H\). Let \(X\in \fn\) be such that \(\Tilde{n}'(0) = DL_n(e)X\) and \(\xi\coloneqq \Tilde{H}'(0) \in \fa\). Then we get
\begin{align*}
Y(t) &= tDL_n(\exp(tH))D\exp(tH)\Tilde{H}'(0) + DR_{\exp(tH)}(n)\Tilde{n}'(0)\\
&= tDL_n(\exp(tH))DL_{\exp(tH)}(e)\xi + DR_{\exp(tH)}(n)DL_n(e)X\\
&= tDL_{n\exp(tH)}(e)\xi + DL_{n\exp(tH)}(e)\Ad(\exp(-tH))X\\
&= DL_{c(t)}(e)\left[t\xi + \Ad(\exp(-tH))X\right]
\end{align*}
where we applied the formula for the derivative of the exponential map in the second step and 
\begin{align*}
DR_{\exp(tH)}(n)DL_n(e) &= DL_n(\exp(tH))DR_{\exp(tH)}(e)\\
&= DL_n(\exp(tH))DL_{\exp(tH)}(e)\Ad(\exp(-tH))\\
&= DL_{n\exp(tH)}(e)\Ad(\exp(-tH)).
\end{align*}
in the third step. This finishes the proof.
\end{proof}

Before we continue we need some more facts about the structure theory of \(G\) from Azencott and Wilson \cite{azencott_wilson_1}, \cite{azencott_wilson_2}. We adopt the formulation of Heber in Proposition 4.2 in \cite{heber}: There exists a decomposition of \(\fn\) into mutually orthogonal spaces \(\fn_1,\ldots,\fn_m\) which is preserved by every \(\ad(H)\), \(H\in \fa\). For every \(j=1,\ldots,m\) there exist a non-trivial linear functional \(\mu_j: \fa\rightarrow \R\) and a symmetric strictly positive definite linear operator \(D_j\) of \(\fn_j\) such that \(\frac{1}{2}(\ad(H)+\ad(H)^T)N_j = \mu_j(H)D_jN_j\) for every \(N_j\in \fn_j\) and every \(H\in \fa\). There exists \(H_+\in \fa\) such that \(\mu_j(H_+) > 0\) for every \(j=1,\ldots,m\).

\begin{lemma}\label{conelemma}
There exists a constant \(\varepsilon > 0\) and a set \(S \subseteq G\) with non-empty interior such that the following holds: Let \(s\in S\), \(n\coloneqq \pi_N(s)\) and \(c:\R\rightarrow G\), \(c(t) = n\exp(tH)\) where \(H\in \fa\), be the geodesic such that \(c(0) = n\) and \(c(1) = s\). Then \(c([0,\infty[)\subseteq S\) and \(\mu_j(H)\leq -\varepsilon \|H\|\) for all \(j=1,\ldots,m\).
\end{lemma}

\begin{proof}
As there exists an \(H_+\in \fa\) such that \(\mu_j(H_+) > 0\) for every \(j=1,\ldots,m\) it is clear that there exists a constant \(\varepsilon > 0\) and a cone \(C_0\subseteq \fa\) with non-empty interior and apex at the origin, such that \(\mu_j(H) \leq -\varepsilon \|H\|\) for every \(j=1,\ldots,m\) and every \(H\in C_0\). Let \(S\coloneqq \bigcup_{n\in N}n\exp(C_0)\). It is clear that \(\varepsilon\) and \(S\) satisfy the properties of the lemma.
\end{proof}

From now on, \(S\) and \(\varepsilon\) will always denote the set and constant from Lemma \ref{conelemma}.

\begin{lemma}\label{arblarge}
For every \(\rho > 0\) and every compact set \(K\subseteq G\), there exists a left translate of \(K\) which is contained in \(S\ssm (N\cdot B(e,\rho))\).
\end{lemma}

\begin{proof}
Fix \(\rho > 0\). Consider the projection \(\pi_A:G \rightarrow A\), \(\pi_A(s) = a\) for \(s=na\). The set \(\pi_A(K)\) is compact and by how \(S\) is constructed, we can find \(a_0\in A\) such that \(a_0\pi_A(K)\subseteq \pi_A(S)\ssm B(e,\rho)\). Now, \(K\subseteq N\pi_A(K)\) and hence 
\[
a_0K\subseteq a_0N\pi_A(K) = Na_0\pi_A(K) \subseteq S \ssm (N\cdot B(e,\rho))
\]
where we used the normality of \(N\) in the second step.
\end{proof}

Now we turn to establishing contraction properties. For that we must analyse the growth of \(N\)-Jacobi fields. So let \(Y\) be an \(N\)-Jacobi field given by the formula from Lemma \ref{jacobi_field_lemma}. Write \(X = \sum_{j=1}^m X_j\) where \(X_j\in \fn_j\). Using that \(\ad(H)\) preserves the decomposition of \(\fn\), we get
\begin{equation}\label{derivative}
\begin{split}
\frac{1}{2}\frac{d}{dt}\|Y(t)\|^2 &= t\|\xi\|^2 + \sum_{j=1}^m\langle \Ad(\exp(-tH))X_j,-\ad(H)\Ad(\exp(-tH))X_j\rangle\\
&= t\|\xi\|^2 - \sum_{j=1}^m\mu_j(H)\langle \Ad(\exp(-tH))X_j,D_j\Ad(\exp(-tH))X_j\rangle.
\end{split}
\end{equation}

\begin{remark}\label{lipschitz}
We see immediately from \eqref{derivative} that if \(\mu_j(H) \leq 0\) for every \(j=1,\ldots,m\) then \(\|Y(t)\|^2\) and hence \(\|Y(t)\|\) is non-decreasing for \(t \geq 0\). This holds in particular if \(c(1)\in S\) where \(S\) is the set from Lemma \ref{conelemma}. It follows that the projection map \(\pi: M\rightarrow M_0\times N\) is \(1\)-Lipschitz when restricted to \(M_0\times S\).
\end{remark}

Compare the following lemma with Lemma 2.1 in \cite{wolter}.

\begin{lemma}\label{growthlemma}
There exists a constant \(\lambda > 0\) such that the following holds: Let \(c(t)\coloneqq n\exp(tH)\) where \(n\in N\) and \(H\in \fa\), be a geodesic such that \(c(1)\in S\). Let \(Y\) be an \(N\)-Jacobi field along \(c\) of the form 
\[
Y(t) = DL_{c(t)}(e)\Ad(\exp(-tH))X
\]
where \(X\in \fn\). Then
\[
\|Y(t+s)\|\geq e^{\lambda t\|H\|}\|Y(s)\|
\]
for all \(s,t\geq 0\).
\end{lemma}

\begin{proof}
Let \(a > 0\) be a constant such that for each operator \(D_j\), every eigenvalue of \(D_j\) is \(\geq a\). Then \(\langle N_j,D_jN_j\rangle \geq a\|N_j\|^2\) for every \(N_j\in \fn_j\). By Lemma \ref{conelemma}, we have that \(\mu_j(H)\leq -\varepsilon\|H\|\) for every \(\mu_j\). Hence, we get from \eqref{derivative} that
\[
\frac{d}{dt}\|Y(t)\|^2 \geq \sum_{j=1}^m 2a\varepsilon \|H\|\|\Ad(\exp(-tH)X_j\|^2 \geq 2a\varepsilon \|H\|\|Y(t)\|^2.
\]
Multiplying through this inequality by \(e^{-2a\varepsilon\|H\|}\) gives that \(e^{-2a\varepsilon\|H\|}\|Y(t)\|^2\) is non-decreasing in \(t\) so the same holds for \(e^{-a\varepsilon\|H\|}\|Y(t)\|\). Thus \(\|Y(t+s)\|\geq e^{\lambda t\|H\|}\|Y(s)\|\) with \(\lambda \coloneqq a\varepsilon > 0\). 
\end{proof}

\begin{lemma}\label{voldistlemma}
Let \(x = (x_0,s)\in M_0\times S\), \(p\coloneqq \pi(x)\) and let \(c:\R\rightarrow M\) be the geodesic such that \(c(0) = p\) and \(c(1) = x\). Let \(v_1,\ldots,v_k \in T_xM\) be orthonormal vectors which are all orthogonal to \(c'(1)\), \(k\geq r\). Then the map \(\varphi(t,x) = \sigma(t,\pi(x),x)\) satisfies
\[
\det\left(\langle D\varphi(t,x)(0,v_i), D\varphi(t,x)(0,v_j)\rangle \right)^{1/2} \leq e^{-\lambda(1-t)\|H\|}.
\]
\end{lemma}

\begin{proof}
Write \(c(t) = (x_0,c_1(t))\) where \(c_1(t)\coloneqq n\exp(tH)\), \(H\in\fa\). The space of vectors \(DL_{c_1(1)}(e)\Ad(\exp(-H))X\), \(X\in \fn\), intersects the span of \(v_1,\ldots,v_k\) non-trivially, so we may assume that \(v_1\) has this form. For \(i=1,\ldots,k\), let \(Y_i\) be the \(N\)-Jacobi field along \(c_1\) such that \(Y_i(1) = v_i\). Then \(Y_1(t) = DL_{c_1(t)}(e)\Ad(\exp(-tH))X\) for some \(X\in \fn\). Now, by the last lemma, we have \(\|Y_1(t)\|\leq e^{-\lambda(1-t)\|H\|}\|Y(1)\| = e^{-\lambda(1-t)\|H\|}\). For \(i=2,\ldots,k\) we know that \(\|Y_i(t)\|\leq \|Y_i(1)\| = 1\) by Remark \ref{lipschitz}. Thus, using that \(D\varphi(t,x)(0,v_i) = Y_i(t)\), we obtain
\begin{align*}
\det\left(\langle D\varphi(t,x)(0,v_i),D\varphi(t,x)(0,v_j)\rangle \right)^{1/2} &\leq \prod_{i=1}^n\|D\varphi(t,x)(0,v_i)\| = \prod_{i=1}^n\|Y_i(t)\|\\
&\leq e^{-\lambda(1-t)\|H\|}.
\end{align*}
This proves the lemma.
\end{proof}

As a last step before turning to the proof of Theorem 1 we remark the following: Let \(K\subset \R^d\) be a compact set, \(\theta \in L^1(K,\Z)\) and \(F:K\rightarrow M\) a Lipschitz map. Then by (4.12) in \cite{wenger_filling}, the comments thereafter and Rademacher's theorem, we have
\begin{equation}\label{masseq}
\Mass(F_\#[[\theta]])\leq \int_A|\theta(z)|J_dF(z)d\sL^d(z).
\end{equation}
Here \(A\) is the set of points \(z\in K\) such that \(z\) is a Lebesgue point of \(K\) and \(F\) is differentiable at \(z\), \(J_dF(z)\) denotes the \(d\)-dimensional volume distortion factor of \(F\) at \(z\) and \(\sL^d\) denotes the \(d\)-dimensional Lebesgue measure. Further, if \(F\) is bi-Lipschitz, then equality holds in \eqref{masseq} (see (2.22) in \cite{wenger_filling}).

\begin{proof}[Proof of Theorem \ref{main_thm}]
Let \(Z\in \Zc[k](M)\). By Lemma \ref{arblarge}, we can assume that \(\spt(Z)\) is contained in \(M_0\times S\). Let us define \(V_1\coloneqq \varphi_\#([[0,1]]\times Z)\). As in the proof of Theorem 4.1 in \cite{wenger_filling}, we may assume, due to the decomposition theorem for currents, that \(Z = \psi_\#[[\theta]]\) for some bi-Lipschitz \(\psi: K \rightarrow M\) where \(K\subseteq \R^k\) is compact and \(\theta \in L^1(K,\Z)\). By (4.10) in \cite{wenger_filling} we have that \(V_1 = \Tilde{\varphi}_\#[[\Tilde{\theta}]]\) where \(\Tilde{\varphi}(t,z) \coloneqq \varphi(t,\psi(z))\) and \(\Tilde{\theta}(t,z) \coloneqq \theta(z)\). Hence we have by \eqref{masseq} that
\[
\Mass(V_1) \leq \int_{[0,1]\times A} |\theta(z)|J_{k+1}\Tilde{\varphi}(t,z)d\sL^{k+1}(t,z).
\]
Using Lemma \ref{voldistlemma}, it is easy to see that
\[
J_{k+1}\Tilde{\varphi}(t,z) \leq d(\psi(z),\pi(\psi(z)))e^{-\lambda d(\psi(z),\pi(\psi(z)))(1-t)}J_k\psi(z).
\]
Thus, we get
\begin{align*}
\Mass(V_1)& \leq \int_{[0,1]\times A}d(\psi(z),\pi(\psi(z)))e^{-\lambda d(\psi(z),\pi(\psi(z)))(1-t)}|\theta(z)|J_k\psi(z) d\sL^{k+1}(t,z)\\
&\leq \frac{1}{\lambda}\int_A\left(1-e^{-\lambda d(\psi(z),\pi(\psi(z)))}\right)|\theta(z)|J_k\psi(z) d\sL^k(z)\\
&\leq \frac{1}{\lambda}\int_A|\theta(z)|J_k\psi(z)d\sL^k(z)\\
&= \frac{1}{\lambda}\Mass(Z)
\end{align*}
where we used \eqref{masseq} with equality in the last step.\\

Let us now construct a filling \(V_2\) of \(\pi_\#Z\). Using the fact that for any \(x\in M\), the vectors in \(T_xM\) which are tangent to the geodesic from \(\pi(x)\) to \(x\) are sent to zero by \(v\mapsto D\pi(x)v = D\varphi(0,x)(0,v)\), it follows easily from Lemma \ref{voldistlemma} that 
\[
J_k(\pi\circ \psi)(z) \leq e^{-\lambda d(\psi(z),\pi(\psi(z)))}J_k\psi(z).
\]
Thus, we get from \eqref{masseq} that
\begin{align*}
\Mass(\pi_\#Z)&\leq \int_A|\theta(z)|J_k(\pi\circ\psi)(z)d\sL^k(z)\\
&\leq \int_A|\theta(z)|e^{-\lambda d(\psi(z),\pi(\psi(z)))}J_k\psi(z)d\sL^k(z)\\
&\leq \sup_{x\in \spt(Z)}e^{-\lambda d(x,\pi(x))}\,\Mass(Z).
\end{align*}
Since Lemma \ref{arblarge} allows us to assume that \(\spt(Z)\) is as far away from \(M_0\times N\) as wanted, we may assume that \(\Mass(\pi_\#Z) \leq 1\). Let \(V_2\) be a filling of \(\pi_\#Z\) satisfying \(\Mass(V_2)\leq \gamma_k\Mass(\pi_\#Z)^{(k+1)/k}\) where \(\gamma_k\) is the constant for the euclidean isoperimetric inequality for \(k\)-cycles in \(M\) (see Theorem 1.2 in \cite{wenger_euclidean}). Then \(\Mass(V_2)\leq \gamma_k\Mass(\pi_\#Z)^{(k+1)/k}\leq \gamma_k\Mass(\pi_\#Z)\leq \gamma_k\Mass(Z)\) where the last step followed from the fact that \(\pi\) is \(1\)-Lipschitz when restricted to \(M_0\times S\). Now, \(V\coloneqq V_1+V_2\) is a filling of \(Z\) and 
\[
\Mass(V)\leq \Mass(V_1)+\Mass(V_2)\leq \left(\frac{1}{\lambda} + \gamma_k\right)\Mass(Z)
\]
so we are done.
\end{proof}

\begin{remark}\label{symmcase} Let us now show that if \(M\) is a symmetric space, one can take a suitable geodesic cone as a filling (compare the proof of Theorem 1.7 in \cite{wenger_filling}) Recall that a symmetric Hadamard manifold is the product of a euclidean space and a symmetric space of non-compact type. For simplicity, we ignore the euclidean factor. So let \(M=G/K\) be a symmetric space of non-compact type. Let \(\fg = \fk\oplus \fp\) be the associated orthogonal symmetric Lie algebra, \(\fa\subseteq \fp\) a maximal abelian subspace and \(\Sigma\) the set of roots of \(\fg\) with respect to \(\fa\). Consider the set \(C\coloneqq \bigcup_{k\in K} \Exp(\Ad(k)C_0)\). Here \(\Exp: \fp\rightarrow M\) is the exponential map at \(p \coloneqq eK\) and \(C_0\subseteq \fa\) is a closed cone with apex at the origin, chosen together with a constant \(\lambda > 0\) such that \(|\alpha(H)|\geq \lambda \|H\|\) for every \(H\in \fa\) and every \(\alpha\in \Sigma\). It is clear that \(C\) is conical. Further, it has non-empty interior since the map \(K\times \fa \rightarrow \fp\), \((k,H)\mapsto \Exp(\Ad(k)H)\) has a surjective derivative at \((k,H)\) if \(\alpha(H)\neq 0\) for every \(\alpha \in \Sigma\). It follows that \(C\) contains arbitrarily large balls so we may assume that \(\spt(Z)\subseteq C\). Now define \(\varphi\) by \(\varphi(t,x) \coloneqq \sigma(t,p,x)\) for \(t\in[0,1]\), \(x\in M\). Fix \(x\in C\) and let \(c:[0,1]\rightarrow M\), \(c(t) = \Exp(tH)\), be the geodesic from \(p\) to \(x\). Then for every \(t\in [0,1]\) and \(v\in T_xM\) we have that \(D\varphi(t,x)(0,v) = Y(t)\) where \(Y\) is the unique Jacobi field along \(c\) such that \(Y(0)=0\). Every Jacobi field \(Y\) along \(c\) which satisfies \(Y(0) = 0\) can be written as 
\[
Y(t) = \sum_{i=1}^r b_itE_i(t) + \sum_{i=r+1}^nb_i\frac{1}{\lambda_i}\sinh(\lambda_i t)E_i(t)
\]
where \(E_1,\ldots,E_n\) is an orthonormal frame of parallel vector fields along \(c\) and \(\lambda_i > 0\) is such that \(-\lambda_i^2\) is an eigenvalue of the curvature operator \(R_H: \fp\rightarrow \fp\), \(R_H(X) \coloneqq R(X,H)H = -\ad(H)^2X\) (see Proposition 2.15.2(1) in \cite{eberlein}). By how we constructed \(C\), it furthermore holds that \(\lambda_i \geq \lambda\|H\|\). Now it is easy to prove an analogue of Lemma \ref{voldistlemma} and then it follows as in the proof of Theorem \ref{main_thm} that the cone \(V\coloneqq \varphi_\#([[0,1]]\times Z)\) from \(p\) over \(Z\) satisfies \(\Mass(V)\leq \frac{1}{\lambda}\Mass(Z)\).
\end{remark}

\begin{remark}\label{chains}
To prove Theorem 1 for Lipschitz chains, one should proceed similarly as is done in the proof of Theorem 1 in \cite{leuzinger} but project onto \(M_0\times N\) instead of a horosphere, and then use the volume contraction estimates we used above, instead of the ones from \cite{leuzinger}. To fill the projected cycle, one should use Gromov's filling inequality (2.3 in \cite{gromov_filling}).
\end{remark}

%% file: acknowledgements.tex
\noindent 
{\bf Acknowledgements } I want to thank Urs Lang for suggesting that I try to prove the result presented in this paper and for introducing the relevant literature to me. I also thank him for corrections to the manuscript and good suggestions. Finally, I want to express my gratitude for financial support from the Swiss National Science Foundation.

%% file: main.bbl
\begin{thebibliography}{9}

\bibitem{azencott_wilson_1}
Azencott, R. and Wilson, E.: Homogeneous manifolds of negative curvature. I. Trans. Amer. Math. Soc. {\bf 215}, 323-362 (1976)

\bibitem{azencott_wilson_2}
Azencott, R. and Wilson, E.: Homogeneous manifolds of negative curvature. II. Mem. Am. Math. Soc. {\bf 178} (1976)

\bibitem{eberlein}
Eberlein, P.: Geometry of nonpositively curved manifolds. Chicago Lectures in Math., University of Chicago Press, Chicago, IL (1996)

\bibitem{gromov_asympt}
Gromov, M.: Asymptotic invariants of infinite groups. In: Niblo, A., Roller, M.A. (eds.) Geometric Group Theory. London Mathematical Society Lecture Note Series, pp. 1–295. Cambridge University Press, Cambridge (1993)

\bibitem{gromov_filling}
Gromov, M.: Filling Riemannian manifolds. J. Differ. Geom. {\bf 18}, 1-147 (1983)

\bibitem{heber}
Heber, J.: On the geometric rank of homogeneous spaces of nonpositive curvature. Invent. Math. {\bf 112}, 151-170 (1993)

\bibitem{kl_hrh}
Kleiner, B., Lang U.: Higher rank hyperbolicity. Invent. Math. {\bf 221}, 597-664 (2020)

\bibitem{leuzinger}
Leuzinger, E.: Optimal higher-dimensional Dehn functions for some CAT(0) lattices. Groups Geom. Dyn. {\bf 8}, 441–466 (2014)

\bibitem{wenger_euclidean}
Wenger, S.: Isoperimetric inequalities of Euclidean type in metric spaces. Geom. Funct. Anal. {\bf 15}, 534–554 (2005)

\bibitem{wenger_filling} 
Wenger, S.: Filling invariants at infinity and the Euclidean rank of Hadamard spaces. Int. Math. Res. Not. {\bf 16}, 1-33 (2006)

\bibitem{wolter}
Wolter, T.: Geometry of homogeneous Hadamard manifolds. Int. J. Math. {\bf 2}, 223-234 (1991)

\end{thebibliography}
